\documentclass[11pt,a4paper]{article}
\usepackage{amsmath,amssymb,amstext,amsfonts,latexsym,amsthm,mathtools}
\usepackage{dsfont,mathrsfs,txfonts}
\usepackage{graphicx,float,hyperref,color}
\usepackage[affil-sl]{authblk}
\usepackage[english,activeacute]{babel}
\usepackage{inputenc}
\usepackage[shortlabels]{enumitem}% itemise noident ([leftmargin=*])
\usepackage{orcidlink}
\newtheorem{theorem}{Theorem}
\newtheorem{lemma}{Lemma}

\newtheorem{remark}{Remark}

\newcommand{\CC}{\mathds{C}}
\newcommand{\NN}{\mathds{N}}
\newcommand{\PP}{\mathds{P}}
\newcommand{\RR}{\mathds{R}}
\newcommand{\RRp}{\mathds{R}_+}
\newcommand{\ZZ}{\mathds{Z}}
\newcommand{\ZZp}{\mathds{Z}_+}

\newcommand{\dsty}{\displaystyle}

\newcommand{\unifn}{\;{\mathop{\rightrightarrows}_{n}}\;}
\newcommand{\med}{{\mathfrak{M}^{\prime}[\RRp]\/}}

\title{\bf Asymptotic  for orthogonal polynomials with respect to a rational modification of a measure supported on the semi-axis}

\author[1]{Carlos F\'{e}liz-S\'{a}nchez\,\orcidlink{0000-0003-1824-4583} \thanks{cfeliz79@uasd.edu.do}\thanks{The research of C. F\'{e}liz-S\'{a}nchez  was   partially supported by    Fondo Nacional de Innovaci\'{o}n  y Desarrollo Cient\'{\i}fico y Tecnol\'{o}gico (FONDOCYT),  Dominican Republic, under~grant   2020-2021-1D1-136.}}
\author[2]{H\'{e}ctor Pijeira-Cabrera\,\orcidlink{0000-0002-8465-0646}\thanks{hpijeira@math.uc3m.es}}
\author[3]{Javier Quintero-Roba\,\orcidlink{0000-0003-0536-3102}\thanks{javier.quintero@urjc.es}}

\affil[1]{Instituto de Matem\'{a}ticas, Facultad de Ciencias, Universidad Aut\'{o}noma de Santo Domingo, \linebreak Santo Domingo 10105, Dominican Republic.}

\affil[2]{Departamento de Matem\'{a}ticas, Universidad Carlos III de Madrid,  \linebreak  Legan\'{e}s 28911,  Madrid,  Spain.}
\affil[3]{Departamento de Teor\'{\i}a de la Se\~{n}al y Comunicaciones  y Sistemas Telem\'{a}ticos y Computaci\'{o}n, Universidad Rey Juan Carlos,  \linebreak Fuenlabrada, 28943,  Madrid,  Spain.}

\date{}

\begin{document}

%%%%%%%%%%%%%
\maketitle

%%%%%%%%%%%%%%%%%%%%%%%%%%%%%%%%%%%%%%%%%%%%%%%%
\begin{abstract}
Given a sequence of orthogonal polynomials $\{L_n\}_{n=0}^\infty$, orthogonal with respect to a positive Borel $\nu$ measure supported on $\mathbb{R}_+$, let $\{Q_n\}_{n=0}^\infty$ be the sequence of orthogonal polynomials with respect to the modified measure $r(x)d\nu(x)$, where $r$ is certain rational function, and {$L_n(-1) = Q_n(-1)= (-1)^n$}. This work is devoted to the proof of the relative asymptotic
$$
\frac{Q_n^{(d)}(z)}{L_n^{(d)}(z)} \unifn  \prod_{k=1}^{N_1}\left(\frac{\sqrt{a_k}+i}{\sqrt{z}+\sqrt{a_k}}\right)^{A_k}\prod_{j=1}^{N_2} \left(\frac{\sqrt{z}+\sqrt{b_j}}{\sqrt{b_j}+i}\right)^{B_j},$$
on compact subsets of $\mathbb{C}\setminus\mathbb{R}_+$,  where $a_k$ and $b_j$ are the zeros and poles of $r$, and the $A_k$, $B_j$ are their respective multiplicities.

\bigskip



\end{abstract}

%%%%%%%%%%%%%%%%%%%%%%%%%%%%%%%%%%%%%%%%%%%%%%%%%%%%%%%%%%%%%%%%%%%%%%%%%%%%%%%%%%%%%%%%%

\section{Introduction}

Let $\mu$ be a  positive, finite, Borel measure on   $\RRp=[0,+\infty)$, such that for all  $ n\in \ZZp$ (the set of all non-negative integers)
\begin{equation}\label{mu-Moments}
\eta_n=\int_{0}^{\infty} x^n \, d\mu(x) < \infty.
\end{equation}
there is no other measure $\mu_0$, such that  $\dsty \eta_n=\int_{0}^{\infty} x^n \, d\mu_0(x)$ for all  $ n\in \ZZp$,  it is said that the moment problem associated with $\{\eta_n\}_{n\in \ZZp}$  is determined (see (\cite{Sch17} Ch. 4)). By~a classical result of T. Carleman  (see (\cite{Sch17} Th.  4.3)), a~sufficient condition in order to the moment problem associated with the sequence $\{\eta_n\}_{n\in \ZZp}$  in \eqref{mu-Moments} to be determined is
\begin{equation}\label{Carleman-Cond}
\sum_{n=1}^{\infty} \frac{1}{\sqrt[2n]{\eta_n}}=+ \infty.
\end{equation}
We  say that the measure  $\mu$ belongs to the class $\med$ if $\{\eta_n\}_{n\in \ZZp}$  satisfies \eqref{Carleman-Cond} and  $\mu^{\prime}>0$   a.e. on $\RRp$ with respect to Lebesgue~measure.

Let $\dsty r(z)= \frac{\alpha(z)}{\beta(z)}$  be a rational function, where $\alpha$ and $\beta$ are coprime polynomials with  respective degrees  $A$ and $B$.  We say that $d\mu_{r}(x)= r(z) d \mu(z)$ is a rational modification (for brevity, modification) of the measure $\mu$. Write
$$\alpha(z)=\prod_{i=1}^{N_1}(z-a_i)^{A_i}, \qquad \beta(z)=\prod_{j=1}^{N_2}(z-b_j)^{B_j}, $$
where $ a_i,b_j\in \CC\setminus \RRp$, $A_i,B_j\in \NN$. $A=A_1+\cdots+A_{N_1}$ and $B=B_1+\cdots+B_{N_2}.$

We denote by $\{L_n\}_{n=0}^\infty$ the sequence of monic orthogonal polynomials with respect to $d\mu$. Assume that $\{Q_n\}_{n=0}^\infty$ is the sequence of monic polynomials of least degree, not identically equal to zero, such that
\begin{equation}\label{PerturbedMonicPoly}
\int_{0}^{\infty} x^k \,  Q_n(x)\;r(x) d\mu(x)=0, \quad \text{for all} \quad  k=0,1,2,\dots,n-1.
\end{equation}
The existence of $Q_n$ is an immediate consequence of \eqref{PerturbedMonicPoly}.  Indeed, it is deduced solving an homogeneous linear system with $n$ equations and $n+1$ unknowns. Uniqueness follows from the minimality of the degree of the polynomial. We call $Q_n$ the $n$th monic modified orthogonal polynomial. In~(\cite{uvarov69} Th.1),  explicit formulas are provided in order to compute $Q_n$ when the poles and zeros of the rational modification have a multiplicity of~one.%Check meaning retained

Suppose that $\{a_i\}_{i=1}^{N_1},\{b_j\}_{j=1}^{N_2} \subset \CC \setminus [-1,1] $.  If~$\mu$ is a  positive  (finite Borel) measure on   $[-1,1]$, such that $\mu$ is on the Nevai class $\mathfrak{M}(0,1)$, in~ (\cite{LopMarWal95} Th. 1) the authors prove the following asymptotic formula
\begin{equation}\label{AsymBoundedCase}
\frac{Q_n^{(d)}(z)}{L_n^{(d)}(z)} \unifn  \prod_{i=1}^{N_1} \left( \frac{\varphi(z) - \varphi(a_i)}{2 (z- a_i)}\right)^{A_i}\;  \prod_{j=1}^{N_2} \left( 1- \frac{1}{\varphi(z) \varphi(b_j)}\right)^{B_j},
\end{equation}
on $K\subset \overline{\CC} \setminus [-1,1]$. The~notation  $ f_n \unifn f, \;K \subset U$ means that  the sequence of functions $f_n$ converges to $f$ uniformly on  a compact subset $K$ of the region $U$,   $f^{(d)}$ denotes the $d$th derivative of $f$,  $d\in \ZZp$ is fixed  and
$$\varphi(z)=z+\sqrt{z^2-1}\quad \left( \left| z+\sqrt{z^2-1} \right|>1, \quad z \in \CC \setminus [-1,1] \right).$$

In~\cite{LopMarWal95}, the~asymptotic formula \eqref{AsymBoundedCase} is pivotal in examining the asymptotic properties of orthogonal polynomials across a broad range of inner products, encompassing Sobolev-type inner products  $$\langle f,g \rangle_{S} = \int f g \, d\mu
+ \sum_{j=1}^m \sum_{i=0}^{d_j} \lambda_{j,i} \; f^{(i)}(\zeta_j) \,g^{(i)}(\zeta_j),$$
where $\lambda_{j,i} \geq  0$,  $m, d_j > 0$, $\mu$ is certain kind of complex measure with compact support is defined on the real line, and~$\zeta_j$ represents complex numbers outside the support of $\mu$. The~authors compare the Sobolev-type orthogonal polynomials associated with this measure to the orthogonal polynomials with respect to $\mu$. These asymptotic results are of interest for the electrostatic interpretation of zeros of Jacobi--Sobolev polynomials (cf.~\cite{PiQuiTo23}).

On the other hand, the~use of modified   measures  provides a stable way of computing the  coefficients of the recurrence relation associated to a family of orthogonal polynomials (see (\cite{Gau04}~Ch. 2)) and  in~\cite{Lago85,Lago89} the interest of the modified orthogonal polynomials  for  the study of the multipoint Pad\'{e} approximation is~shown.

For measures supported on $[0,+\infty)$ (or $(-\infty,+\infty)$) that satisfy the Carleman condition, G. L\'{o}pez in (\cite{Lago90} Th. 4) (or (\cite{Lago90} Th. 3)  for $(-\infty,+\infty)$) proves a quite general version of the relative asymptotic formula \eqref{AsymBoundedCase}. In~this case, if~the modification function, $\rho$,  is a non-negative function on $[0,+\infty)$ in $\mathrm{L}^1(\mu)$, such that there exists an algebraic polynomial $G$ and $k \in \NN$ for which $|G|\rho/(1+x)^k$ and $|G|\rho^{-1}/(1+x)^k$  belong to $\mathrm{L}^{\infty}(\mu)$, then
\begin{equation}\label{AsymUNBoundedCase}
\frac{Q_n(z)}{L_n(z)} \unifn    \frac{\mathrm{S}(\rho,{\CC} \setminus [0,+\infty),z )}{\mathrm{S}(\rho,{\CC} \setminus [0,+\infty),\infty )}, \quad K \subset {\CC} \setminus [0,+\infty);
\end{equation}
where  $\mathrm{S}(\rho,{\CC} \setminus [0,+\infty),z )$ is the Szeg\H{o}'s function for $\rho$ with respect to ${\CC} \setminus [0,+\infty)$, i.e.,
\begin{align*} %\label{SzegoFunction}
 \mathrm{S}(\rho,{\CC} \setminus [0,+\infty),z )= &  e^{ s(z)},  \quad s(z)=\frac{1}{2 \pi}\int_{0}^{\infty} \log\rho(x) \left( \frac{\sqrt{-z}}{z-x}\right) \frac{dx}{\sqrt{x}};\\ \nonumber
  \mathrm{S}(\rho,{\CC} \setminus [0,+\infty),\infty )= & \lim_{r \to +\infty}  \mathrm{S}(\rho,{\CC} \setminus [0,+\infty),-r );
\end{align*}
where the  roots are selected from the condition $\sqrt{1}=1.$ Additionally,  it is requested that $f(z)=\rho(-((z+1)/(z-1))^2)$ satisfies the  Lipschitz condition in $z=1$ and $f(1)\neq 0$.

Asymptotic results, analogous to those obtained in~\cite{LopMarWal95}, are obtained in~\cite{DiHerPi23}  for the particular case of \eqref{AsymUNBoundedCase}, when $d\mu(x)= x^{a} e^{-x} dx$ with $a >-1$ (the Laguerre measure).

The aim of this paper is to obtain an analog of \eqref{AsymBoundedCase} for measures supported on $\RR_+$, {using a different normalization}. We prove the following~theorem.

\begin{theorem}\label{main:result:CH}
Given a measure $\nu\in \med$, {let $Q_n$, and $L_n$ be the $n$th degree orthogonal polynomials with respect to the measures $r(x)d\nu$ and $d\nu$, normalized with the condition $L_n(-1) = Q_n(-1) = (-1)^n$. Then} it holds in compact subsets of $\CC\setminus \RR$
\begin{equation}\label{final:formula}
\frac{Q_n^{(d)}(z)}{L_n^{(d)}(z)} \unifn  \prod_{i=1}^{N_1}\left(\frac{\sqrt{a_i}+i}{\sqrt{z}+\sqrt{a_i}}\right)^{A_i}\prod_{j=1}^{N_2} \left(\frac{\sqrt{z}+\sqrt{b_j}}{\sqrt{b_j}+i}\right)^{B_j},
\end{equation}
for $d\in \ZZp$.
\end{theorem}

This situation  is not a particular case of \eqref{AsymUNBoundedCase},  because~we consider $\rho$ as a rational function with complex coefficients and no necessarily  $\rho(x) \geq 0$   on $\RRp$.

The structure of the paper is as follows:  Sections~\ref{Ch2-Sec2}  and \ref{Sec-3}  are devoted to prove some preliminary results on varying measures. On~the other hand, in~Section~\ref{Ch2-Sec3}  we obtain an essential theorem    that allows us to finally prove Theorem \ref{main:result:CH} in Section~\ref{sec:rationalAsym}.

%%%%%%%%%%%%%%%%%%%%%%%%%%%%%%%%%%%%%%%%%%%%%%%%%%%%%%%%%%%%%%%%%%%%%%%%%%%%%%%%%%%%%%%

\section{Varying measures and Carleman's~condition} \label{Ch2-Sec2}

In this section, we introduce auxiliary results on varying measures and prove some useful lemmas that allow  us to extend results that hold for measures with bounded support to the unbounded case. The~following notations will be used throughout the paper:
\begin{equation}\label{ChangeVariables}
\begin{aligned}
\Psi(z)=&  \frac{1+z}{1-z}  \;\text{ for }  \;  z \in \CC \setminus [-1,1]. \\
 \Psi^{-1}(z)=&  \frac{z-1}{z+1} \;\text{ for } \;  z \in \CC \setminus \RRp.\\
      \Phi(z)= & \frac{\sqrt{z}+i}{\sqrt{z}-i}  \; \text{ where } \;   \Phi(-1)=\infty \; \text{ and }  \; z \in \CC \setminus [|z|\leq1].
      \end{aligned}
\end{equation}
If  $\sigma$ is  a finite positive Borel measure on $[-1,1]$, we denote
\begin{equation}\label{VaryingMeasure}
d\sigma_n(t)= \frac{d\sigma(t)}{(1-t)^{2n}} \quad \text{ and }\quad  \varsigma_{n}=\int_{-1}^{1} \frac{d\sigma(t)}{(1-t)^n}.
\end{equation}

Here, we consider the principal branch of the square root, i.e.,~$\sqrt{re^{i\theta}} = \sqrt{r}e^{i\frac{\theta}{2}}$, where $r>0$ and $0\leq \theta < 2\pi$.

\begin{lemma}\label{Carleman-Equiv} Let $\mu$ be a positive Borel measure supported on $\RRp$ and suppose that  $\dsty d\sigma(t)= (1-t) \, d\mu\left(\Psi(t)\right)$.~Then,
\begin{enumerate}
 \item[(a)]  $\mu^{\prime}>0$ a.e. on $\RRp$ implies that   $\sigma^{\prime}>0$ a.e. on $[-1,1]$,
  \item[(b)] if $\dsty\sum_{n=1}^{\infty} \frac{1}{\sqrt[2n]{{\eta}_n}}=+ \infty$, then $ \dsty \sum_{n=1}^{\infty} \frac{1}{\sqrt[2n]{\varsigma_n}}=+ \infty,$
\end{enumerate}
where, as~in \eqref{mu-Moments},   $\eta_n$ denotes the $n$th moment of the measure $d\mu$.
\end{lemma}

\begin{proof}  To prove the first assertion note that  if  $\dsty  d\sigma(t)=  \frac{2}{1-t} \,\mu^{\prime}\left(\Psi(t)\right) dt$, then

\begin{equation*}
\frac{d\sigma}{dt}= (1-t) \frac{d\mu\left(\Psi(t)\right)}{dt}= \frac{2}{1-t}\mu'\left(\Psi(t)\right) >0 \quad \text{ a.e. on } [-1,1].
\end{equation*}
The second part is derived using the change of variable $\dsty t= \Psi^{-1}(x)$ in the integral
\begin{align}\nonumber
\varsigma_n & = \int_{-1}^1 \frac{(1-t)}{(1-t)^n}d\mu\left(\Psi(t)\right) = \int_{0}^\infty \left(\frac{x+1}{2}\right)^{n-1} d\mu(x)
\\ \nonumber & = \int_{0}^1 \left(\frac{x+1}{2}\right)^{n-1} d\mu(x) + \int_{1}^\infty \left(\frac{x+1}{2}\right)^{n-1} d\mu(x)\\
& \leq \eta_0 + \int_{1}^\infty x^{n-1} d\mu(x) \leq \eta_0+ \eta_n. \label{CritComparacion}
\end{align}
As  $\dsty \sum_{n=0}^\infty (\eta_n)^{-1/2n}= +\infty$, it follows from~\eqref{CritComparacion} that assertion~(b) holds.
\end{proof}

\begin{lemma}\label{Lemma:factor(x+1)} Assume that   $d\nu \in \med$,   $\dsty r_k(x)= \left(\frac{x+1}{2}\right)^k$ and consider the modification $d\nu_{r_k}(x)=r_k(x) d\nu(x)$. Then  $\dsty d\nu_{r_k}(x)\in \med$ \: for all $k\in \ZZ$.
\end{lemma}
\begin{proof} {We now proceed by induction. Obviously, the~initial  case $k=0$ is given by~hypothesis.}

\begin{itemize}
\item \emph{Case $k>0$.} Assume that   $\dsty d\nu_{r_j}(x)\in \med$ for all $j\leq k-1$. Since $\dsty d\nu_{r_{k}}(x)=\left(\frac{x+1}{2}\right) d\nu_{r_{k-1}}(x)$, it is immediate that  $\dsty d\nu_{r_k}(x)$ is positive and $\frac{d\nu_{r_k}(x)}{dx}>0$ a.e. on $\RRp$.
\end{itemize}

Let $m_{n,k}$ be the  $n$th moment of the measure $\dsty d\nu_{r_{k}}(x)$, then
\begin{align*}
m_{n,k} & = \int_0^\infty x^n   d\nu_{r_{k}}(x) = \int_0^1 x^n \left(\frac{x+1}{2}\right) d\nu_{r_{k-1}}(x)+\int_1^\infty x^n \left(\frac{x+1}{2}\right)  d\nu_{r_{k-1}}(x),\nonumber\\
& \leq  \int_0^1   d\nu_{r_{k-1}}(x)+\int_1^\infty x^{n+1}   d\nu_{r_{k-1}}(x) \leq m_{0,k-1}+ m_{n+1,k-1},%\label{inequality:choosen}
\end{align*}
where we use that $\dsty x^n \left(\frac{x+1}{2}\right)\leq 1$ for  $\dsty x\in [0,1]$ and $\dsty  \left(\frac{x+1}{2}\right)\leq x$, for~ $x\in [1,+\infty)$. Then, using induction hypothesis, we obtain that $m_{n,k}< \infty$ and the sequence of moments for $\dsty  d\nu_{r_{k}}(x)$  satisfies Carleman's~condition.

\begin{itemize}
\item \emph{Case $k<0$.} Repeating the previous arguments, we obtain  that   if $\dsty d\nu_{r_{j}}(x)\in \med$ for all $0 < j\leq k+1$ then $\dsty d\nu_{r_{k}}(x)$ is positive and $\frac{d\nu_{r_{k}}(x)}{dx}>0$ a.e. on $\RRp$.
\end{itemize}

For the   $n$th moment of the measure $\dsty d\nu_{r_{k}}(x)$, we have
\begin{align*}
m_{n,k} & = \int_0^\infty x^n  d\nu_{r_{k}}(x)=  \int_0^1 x^n \left(\frac{2}{x+1}\right) d\nu_{r_{k+1}}(x)+\int_1^\infty x^n \left(\frac{2}{x+1}\right) d\nu_{r_{k+1}}(x)\\
&  \leq   \, 2 \, m_{0,k+1} + m_{n,k+1},\nonumber
\end{align*}
where we use that $\dsty x^n \left(\frac{2}{x+1}\right)\leq 2$ for  $\dsty x\in [0,1]$ and $\dsty  \left(\frac{2}{x+1}\right)\leq 1$, for~ $x\in [1,+\infty)$. Then, using induction hypothesis, we obtain that $m_{n,k}< \infty$ and the sequence of moments for $\dsty  d\nu_{r_{k}}(x)$  satisfies Carleman's condition.
\end{proof}

\begin{lemma}{\cite[Th. 4, Cor. 1]{Lago89}} \label{Lago-1} Let  $P_{n,k}$ be the $k$th monic orthogonal polynomial with respect to $d\sigma_{n}$. If~ $\sigma^{\prime}>0$ a.e. on $[-1,1]$ and  $ \dsty \sum_{n=1}^{\infty} \frac{1}{\sqrt[2n]{\varsigma_n}}=+ \infty,$ then, for~each integer $k$
\begin{equation*}%\label{AsimpRelativ-Varying}
 \frac{P_{n,n-k+1}}{P_{n,n-k}}(z) \unifn \frac{\varphi(z)}{2}; \quad K\subset \CC \setminus [-1,1],
\end{equation*} where $\varphi(z)= z+\sqrt{z^2-1}$ $ \left( \left|  z+\sqrt{z^2-1} \right|>1 \quad z \in \CC \setminus [-1,1] \right)$.
\end{lemma}

\begin{lemma}\label{Th-AsympCocient}  Assume  $\mu \in \med$ and $\dsty  d\mu_{m}(x)=\left(\frac{2}{x+1}\right)^{2m}d\mu(x)$, with~$m \in {\ZZ}$.

\begin{enumerate}
  \item[(a)] Let $\ell_{m,n}$  be the $n$th orthogonal polynomial with respect to $\dsty  \mu_{m}$, normalized by the condition $\ell_{m,n}(-1)=(-1)^n$, then for $d \in \ZZp$, on~$ K \subset \CC \setminus \RRp$ it holds
\begin{equation}\label{AsimpRatioOP}
\frac{\ell_{m,n+m}^{(d)}(z)}{\ell_{k,n+k}^{(d)} (z)}\unifn \left(\frac{z+1}{4}\right)^{m-k} \Phi^{m-k}(z)= \left(\frac{\sqrt{z}+i}{2}\right)^{2(m-k)} .
\end{equation}

  \item[(b)] Let  $L_{m,n}$ be the $n$th monic orthogonal polynomial with respect to $\dsty \mu_{m}$, then on $ K \subset \CC \setminus \RRp$ it~holds
\begin{equation}\label{AsimpRatioOP-Monic}
\frac{L_{m,n+m}^{(d)}(z) }{L_{k,n+k}^{(d)}(z) }\unifn \left(z+1\right)^{m-k} \Phi^{m-k}(z)= \left(\sqrt{z}+i\right)^{2(m-k)}.
\end{equation}
\end{enumerate}

\end{lemma}

\begin{proof}

\emph{(Proof of a).}  Taking  $\dsty  d\sigma_n(t)= (1-t)^{1-2n} d\mu\left(\Psi(t)\right)$, from~the assumptions and Lemma~\ref{Carleman-Equiv},  we obtain that $d\sigma_{n}$ is a finite positive Borel measure on $[-1,1]$, $\sigma_{n}^{\prime}>0$ a.e. on $[-1,1]$ and $\dsty \sum_{n=1}^{\infty} \varsigma_n^{-1/(2n)}=+ \infty$, where $\varsigma_n$ is as in \eqref{VaryingMeasure}. {Since we are interested in the asymptotic formula, we can assume $n+m>1$.}

Let  $P_{n,k}$ be the $k$th monic orthogonal polynomial with respect to $d\sigma_{n}$ and  denote $\dsty \ell^*_{m,n+m}(z)=\left(\frac{z+1}{2}\right)^{n+m} P_{n,n+m}\left(\Psi^{-1}(z)\right)$. After~a change of variable $\dsty x=\Psi(t)$ in the next  integral,  we obtain
\begin{align}\nonumber
 \int_{0}^{\infty}  \left(\frac{x+1}{2}\right)^{k}  \ell^*_{m,n+m}(x)d\mu_m(x)  & = \int_{-1}^{1}\frac{1}{\left(1-t\right)^{n+m+k} } \;  P_{n,n+m}\left(t\right) \left(1-t\right)^{2m} d\mu\left(\Psi(t)\right)  \\ \nonumber
& =  \int_{-1}^{1}  \left({1-t}\right)^{n+m-1-k}   P_{n,n+m}\left(t\right) \; \frac{d\mu\left(\Psi(t)\right)}{(1-t)^{2n-1}}\\ \label{AsimpRatioOP-01}
 = &  \int_{-1}^{1}  \left({1-t}\right)^{n+m-1-k}   P_{n,n+m}\left(t\right) \; d\sigma_{n}\left(t\right)=0,\\ \nonumber
 & \quad \text{for } k= 0,1,\cdots, n+m-1.\\ \label{AsimpRatioOP-02}
  \ell^*_{m,n+m}(-1)= & \lim_{z \to -1} \left(\frac{z+1}{2}\right)^{n+m} P_{n,n+m}\left(\Psi^{-1}(z)\right)=(-1)^{n+m}.
\end{align}
From \eqref{AsimpRatioOP-01} and \eqref{AsimpRatioOP-02}, we have $\ell_{m,n+m}=\ell^*_{m,n+m}$. Therefore,
\vspace{6pt}
\begin{align}
\ell_{m,n+m}(z)& =\left(\frac{z+1}{2}\right)^{n+m} P_{n,n+m}\left(\Psi^{-1}(z)\right), \label{ConectionFormula02}\\
 \frac{\ell_{m,n+m}(z) }{(1+z)^{m-k}\, \ell_{k,n+k}(z) } & =\frac{ P_{n,n+m}\left(\Psi^{-1}(z)\right)}{2^{m-k} P_{n,n+k}\left(\Psi^{-1}(z)\right)}\nonumber\\
 & =  \frac{1}{2^{m-k}}  \prod_{j=k}^{m-1}\frac{ P_{n,n+j+1}\left(\Psi^{-1}(z)\right)}{P_{n,n+j}\left(\Psi^{-1}(z)\right)}.\nonumber
\end{align}

From Lemma \ref{Lago-1}, for~$j=k,\ldots,m-1$;
$$
\frac{ P_{n,n+j+1}\left(\Psi^{-1}(z)\right)}{P_{n,n+j}\left(\Psi^{-1}(z)\right)} \unifn \frac{\varphi\left(\Psi^{-1}(z)\right)}{2}; \quad K\subset \CC \setminus\RRp.
$$
Thus,
\begin{equation*}%\label{AsimpRatioOP-03}
\frac{\ell_{m,n+m}(z) }{\ell_{k,n+k}(z)}\unifn \left(\frac{z+1}{4}\right)^{m-k} \varphi^{m-k}\left(\Psi^{-1}(z)\right); \quad K\subset \CC \setminus\RRp,
\end{equation*}
which  establishes  \eqref{AsimpRatioOP}  for $d=0$. In~order to proof \eqref{AsimpRatioOP}  for $d>0$, we proceed by induction on $d$.
\begin{equation*}
\frac{\ell_{m,n+m}^{(d+1)}(z)}{\ell_{k,n+k}^{(d+1)}(z)}= \frac{\ell_{m,n+m}^{(d)}(z)}{\ell_{k,n+k}^{(d)}(z)}+ \frac{\ell_{k,n+k}^{(d)}(z)}{\ell_{k,n+k}^{(d+1)}(z)}\cdot \left(\frac{\ell_{m,n+m}^{(d)}(z)}{\ell_{k,n+k}^{(d)}(z)}\right)^{\prime}.%\label{AsimpRatioOP-05}
\end{equation*}
Assume that formula \eqref{AsimpRatioOP} holds  for $d \in \ZZp$, then  $\dsty \left({\ell_{m,n+m}^{(d)}}/{\ell_{0,n}^{(d)}}\right)^{\prime}$ is uniformly bounded on  compact subsets $K \subset \CC \setminus \RRp$. Note that $\dsty {\ell_{0,n}^{(d)}}/{\ell_{0,n}^{(d+1)}}\unifn 0$ on $ K\subset \CC \setminus\RRp$. This is  proved using an analogous of (\cite{LopMarWal95} (2.9)), and~the Bell's polynomials version of the Faa Di Bruno formula, see (\cite{johnson03} pp. 218, 219).  The~assertion (a) is~proved.

\emph{(Proof of b).} Write $\dsty f_{d,m,n}(z)=\frac{\ell_{m,n+m}^{(d)}(z) }{z^m \; \ell_{0,n}^{(d)}(z) }$ and let $\kappa_{m,n+m}$ be  the leading coefficient of $\ell_{m,n+m}$.
Hence,   for~$d>1$
\begin{align*}
f_{d,m,,k,n}(\infty)& =\frac{(n+m)\cdots(n+m-d+1)\kappa_{m,n+m}}{(n+k)\cdots(n+k-d+1)\kappa_{k,n+k}} \\
f_{0,m,k,n}(\infty)& =\frac{\kappa_{m,n+m}}{\kappa_{k,n+k}}.
\end{align*}
   From \eqref{AsimpRatioOP},
\begin{align}\label{AsimpRatioOP-Monic-01}
f_{d,m,k,n}(z) & \unifn \left(\frac{z+1}{4z}\right)^{m-k} \Phi^{m-k}(z); \quad  K \subset \overline{\CC} \setminus \RRp, \quad l \in \mathbb{Z_{+}}.\\ \label{AsimpRatioOP-Monic-02}
\lim_{n \to \infty} f_{d,m,k,n}(\infty) & = \lim_{n \to \infty} \frac{\kappa_{m,n+m}}{\kappa_{k,n+k}}=\left(\frac{1}{2}\right)^{2(m-k)}.
\end{align}
As  $L_{m,n+m}^{(d)}(z)= \frac{\ell_{m,n+m}^{(d)}(z)}{\kappa_{m,n+m}}$ for $d\geq 1$, from~\eqref{AsimpRatioOP-Monic-01}--\eqref{AsimpRatioOP-Monic-02}, we get \eqref{AsimpRatioOP-Monic}. \end{proof}

Denote by $\mathfrak{M}[-1,1]$ the class of admissible measures in $[-1,1]$ defined in (\cite{Lago89-1} Sec. 5).
Let $\sigma_n$ a positive varying Borel measure supported on $[-1,1]$ and
$$p_{n,m}(w)=\tau_{n,m}w^m+\cdots,\quad \tau_{n,m}>0$$
be the $m$th orthonormal polynomial with respect to $\sigma_n$, then (\cite{Lago89-1} Th. 7)
\begin{equation}
\lim_{n\to \infty}\frac{\tau_{n,n+k+1}}{\tau_{n,n+k}}=2, \qquad k\in\ZZ. \label{RatioTau}
\end{equation}

\begin{lemma}\label{LemmaIntegral}
Let $\sigma_n$ be an admissible measure, then for all $v\in\ZZ$,
\begin{equation}
\int_{-1}^1 \frac{p_{n,n+v}(t)p_{n,n}(t)}{w-t} d\sigma_n(t)\unifn \frac{1}{\varphi^{|v|}(w)\sqrt{w^2-1}};\qquad  K\subset \CC\setminus [-1,1]. \label{IntegralIdentity}
\end{equation}
\end{lemma}
\begin{proof}

This proof is based on the proof of (\cite{LopMarWal95} Lemma 2). Without~loss of generality, let us consider $v\in \ZZp$. Applying the  Cauchy--Schwarz inequality we have, for~$z\in K\subset\overline{\CC}\setminus [-1,1]$
$$\left|\int_{-1}^{1}\frac{p_{n,n+v}(t)p_{n,n}(t)}{w-t} d\sigma_n(t)\right|\leq \frac{1}{d(K,[-1,1])}<\infty,$$
where $d(K,[-1,1])$ denotes the Euclidian distance between the two sets. Thus, for~ (fixed) values of  $v\in \ZZp$, the~sequence of functions in the left hand side of \eqref{IntegralIdentity} is normal. Thus, we deduce uniform convergence from pointwise convergence. The~pointwise limit follows from (\cite{Lago89-1} Th. 9)
$$\lim_{n\to \infty}\int_{-1}^{1}\frac{p_{n,n+v}(t)p_{n,n}(t)}{w-t} d\sigma_n(t)= \frac{1}{\pi}\int_{-1}^1 \frac{T_v(t)}{w-t}\frac{dt}{\sqrt{1-t^2}},$$
here, $T_v$ is the $vth$ Chebyshev orthonormal polynomial of the first kind. Therefore, \eqref{IntegralIdentity} holds if we prove that
\begin{equation}\label{eq2.13}
\frac{1}{\pi}\int_{-1}^1 \frac{T_v(t)}{w-t}\frac{dt}{\sqrt{1-t^2}}= \frac{1}{\varphi^{v}(w)\sqrt{w^2-1}}.
\end{equation}
{Note that $T_0(t)=1,T_1(t)=t$, and, for~$v\geq 1$},
$$2tT_v(t)=T_{v+1}(t)+T_{v-1}(t),$$
or equivalently
\begin{equation}\label{ChevyshevTTRR}
T_{v+1}=2tT_v-T_{v-1}.
\end{equation}
Next, proceed by induction. Start at $v=0$, expression \eqref{IntegralIdentity}, is obtained from the residue theorem and Cauchy's integral formula. Then, for~$v=1$ we have
\begin{align*}
\frac{1}{\pi} \int_{-1}^1 \frac{T_1(t)}{w-t}\frac{dt}{\sqrt{1-t^2}}= & \frac{w}{\pi}\int_{-1}^1 \frac{1}{w-t}\frac{dt}{\sqrt{1-t^2}}-\frac{1}{\pi} \int_{-1}^1 \frac{dt}{\sqrt{1-t^2}} \\
= & {\frac{w}{\sqrt{w^2-1}}-1} =  \frac{1}{\varphi(w)\sqrt{w^2-1}}.
\end{align*}
Assume \eqref{eq2.13} holds for $v=0,1,\dots,k$; $k\geq 1,$  we will prove that it also holds for $v= k+1$. Combining \eqref{ChevyshevTTRR} and the hypothesis of induction, we~obtain
\vspace{6pt}
\begin{align*}
\frac{1}{\pi} \int_{-1}^{1} \frac{T_{k+1}(t)}{w-t}\frac{dt}{\sqrt{1-t^2}}
 = & \frac{1}{\pi}\int_{-1}^1 \frac{2tT_k(t)}{w-t}\frac{dt}{\sqrt{1-t^2}}-\frac{1}{\pi}\int_{-1}^1 \frac{T_{k-1}(t)}{w-t}\frac{dt}{\sqrt{1-t^2}}\\
= & \frac{2z}{\pi}\int_{-1}^1 \frac{T_k(t)}{w-t}\frac{dt}{\sqrt{1-t^2}}-\frac{1}{\pi}\int_{-1}^1 \frac{T_{k-1}(t)}{w-t}\frac{dt}{\sqrt{1-t^2}}\\
=& \frac{1}{\varphi^{k-1}(w)\sqrt{w^2-1}}\left(\frac{2w}{\varphi(z)}-1\right)\\
=& \frac{1}{\varphi^{k+1}(w)\sqrt{w^2-1}},
\end{align*}
which we wanted to prove.
\end{proof}

\begin{lemma}\label{LemmaOfIntegrals}
Let $d\mu(x)= \left(\frac{x+1}{2}\right)^{A-B}  d\nu(x)$, where $A,B\in \ZZp$, and~$d\nu \in \med$. We have on compact subsets of $ \CC\setminus \RRp$

\begin{multline*}
\displaystyle (v-1)!\tau^2_{n,n-B} \int_0^\infty \left(\frac{x+1}{2}\right)^{k}\frac{\ell_{A-k,n+A-k}(x) \ell_{-B,n-B}(x)}{(x-z)^v} d\nu(x)  \\
\displaystyle \unifn   \left(\frac{-1}{(1+z)\left(2\Phi\left(z\right)\right)^{A+B-k}\sqrt{\left(\Psi^{-1}(z)\right)^2-1}}\right)^{(v-1)}.
\end{multline*}
where $\ell_{n,n+m}$ is defined as in Lemma \ref{Th-AsympCocient}.
\end{lemma}

\begin{proof}
First, the~sequence $\{\ell_{n,n+m}\}_{n\geq 0}$ is well defined because the measure $d\nu \in \med$, implies $d\mu\in \med$ (see Lemma \ref{Lemma:factor(x+1)}).

Let us use the connection formula \eqref{ConectionFormula02} and the change of variable \eqref{ChangeVariables} to obtain
\begin{align*}
& (v-1)! \tau^2_{n,n-B} \int_0^\infty \left(\frac{x+1}{2}\right)^{k}\frac{\ell_{A-k,n+A-k}(x) \ell_{-B,n-B}(x)}{(x-z)^v} d\nu(x),\\
&= (v-1)! \tau^2_{n,n-B} \int_{-1}^1 \frac{P_{n,n+A-k}(t)P_{n,n-B}(t)}{\left(\Psi(t)-z\right)^v}\frac{d\sigma(t)}{(1-t)^{2n+A-B}},\\
f_n^{(v-1)}(z) &=  \frac{(v-1)!\tau_{n,n-B}}{\tau_{n,n+A-k}}\int_{-1}^1 \frac{1}{1-t}\frac{p_{n,n+A-k}(t)p_{n,n-B}(t)}{\left(\Psi(t)-z\right)^v} d\sigma_{n}(t).
\end{align*}
where we use
$$
d\sigma_n(t) = \frac{d\mu(\Psi(t))}{(1-t)^{2n-1}} = \frac{(1-t)^{B-A}d\nu(\Psi(t))}{(1-t)^{2n-1}}
$$
Take the $(v-1)$ primitive with respect to $z$ of the previous expression
\begin{align}
f_n(z)= \frac{\tau_{n,n-B}}{\tau_{n,n+A-k}}\int_{-1}^1 \frac{1}{1-t}\frac{p_{n,n+A-k}(t)p_{n,n-B}(t)}{\Psi(t)-z} d\sigma_{n}(t). \label{IntegralIdentity01}
\end{align}
Since we know that
$$(1-t)(\Psi(t)-z)= (1+z)\left(t-\Psi^{-1}(z)\right),$$
we rewrite \eqref{IntegralIdentity01} as
\begin{align*}
& \frac{\tau^2_{n,n-B}}{1+z}\int_{-1}^1 \frac{P_{n,n+A-k}(t)P_{n,n-B}(t)}{t-\Psi^{-1}(z)} d\sigma_{n}(t),\\
= &\frac{\tau_{n,n-B}}{(1+z)\tau_{n,n+A-k}}\int_{-1}^1 \frac{p_{n,n+A-k}(t)p_{n,n-B}(t)}{t-\Psi^{-1}(z)} d\sigma_{n}(t).
\end{align*}

Then, we use Lemma \ref{LemmaIntegral} and \eqref{RatioTau} to obtain on compact subsets of $\CC\setminus \RRp$,
\begin{multline*}
\frac{\tau_{n,n-B}}{(1+z)\tau_{n,n+A-k}}\int_{-1}^1 \frac{p_{n,n+A-k}(t)p_{n,n-B}(t)}{t-\Psi^{-1}(z)}  d\sigma_{n}(t)\\
\unifn  \left(\frac{-1}{(1+z)\left(2\varphi\left(\Psi^{-1}(z)\right)\right)^{A+B-k}\sqrt{\left(\Psi^{-1}(z)\right)^2-1}}\right)=f(z).
\end{multline*}
Note that by the Cauchy--Schwarz inequality we have for $z\in \CC\setminus \RRp$
\begin{align*}
\left|f_n^{(v-1)}(z)\right|& =\left| \frac{(v-1)!\tau_{n,n-B}}{\tau_{n,n+A-k}}\int_{-1}^1 \frac{1}{1-t}\frac{p_{n,n+B-k}(t)p_{n,n-A}(t)}{\left(\psi(t)-z\right)^v} d\sigma_{n}(t)\right|\\
& \leq \frac{B}{d(K,\RRp)}.
\end{align*}

Then, for~each $v$, the~family $\left\{f_n^{(v-1)}\right\}_n$ is uniformly bounded in each $K\subset\CC\setminus \RRp$, which means by Montel's theorem (c.f. \cite[\S 5.4, Th. 15]{Ahl79}) that $\left\{f_n^{(v-1)}\right\}_{n\geq 0}$ is normal (see (\cite{Ahl79} \S 5.1 Def. 2)), i.e.,~we have that from each sequence $\mathbf{N}\subset \NN$ we can take a subsequence $\mathbf{N}_1\subset\mathbf{N}$ such that
$$f_n^{(v)}\unifn g_{(v)}; \qquad n\in \mathbf{N}_1, \quad K\subset\CC\setminus \RRp.$$
Taking the $(v-1)$ derivative and using the uniqueness of the limit  we obtain
\begin{multline*}
\frac{(v-1)!\tau_{n,n-B}}{\tau_{n,n+A-k}}\int_{-1}^1 \frac{1}{1-t}\frac{p_{n,n+A-k}(t)p_{n,n-B}(t)}{\left(\Psi(t)-z\right)^v} d\sigma_{n}(t) \\
\unifn  \left(\frac{-1}{(1+z)\left(2\Phi\left(z\right)\right)^{A+B-k}\sqrt{\left(\Psi^{-1}(z)\right)^2-1}}\right)^{(v-1)}=f^{(v-1)}(z), % \label{LimitDerivatives}
\end{multline*}
on compact subsets $K\subset \CC\setminus \RRp$, which establishes the formula.
\end{proof}

%%%%%%%%%%%%%%%%%%%%%%%%%%%%%%%%%%%%%%%%%%%%%%%%%%%%%%%%%%%%%%%%%%
\section{Relative asymptotic  within certain class of varying~measures}  \label{Sec-3}

In this section, we obtain the asymptotic relation between orthogonal polynomials with respect to different measures of the class $\left(\frac{x+1}{2}\right)^md\mu(x)$, where $\mu$ is any measure of $\med$ and $m\in \ZZ$. Note that, because~of Lemma \ref{Lemma:factor(x+1)},  the~elements of this class belong to $\med$.

To maintain a general tone in the expositions in this section we use $\mu$ and $\nu$ as two measures in $\med$ having no relation with the previous use of the~notation.

Let $m \in \mathbb{Z}^+$ and $h_{m,n}(z)$  be the $n$th orthogonal polynomial with respect to $\left(\frac{x+1}{2}\right)^md\nu(x)$, normalized as $h_{m,n}(-1)= (-1)^n$. Consider the following relations
\begin{align*}
\int_0^\infty \left(\frac{x+1}{2}\right)^k h_{0,n}(x) d\nu(x) = 0,
\end{align*}
for $k=0,\ldots,n-1$. Apply the change of variable  $\Psi(t) = z$ given in \eqref{ChangeVariables} to obtain
\begin{align*}
0 & = \int_{-1}^1 \left(\frac{1}{1-t}\right)^k h_{0,n}\left(\Psi(t)\right) d\nu\left(\Psi(t)\right)\\
  & = \int_{-1}^1 (1-t)^{n-k-1} (1-t)^nh_{0,n}\left(\Psi(t)\right) \frac{(1-t)d\nu\left(\Psi(t)\right)}{(1-t)^{2n}}.
\end{align*}

The polynomial $H_{n,n}(t) = (1-t)^n h_{0,n}(\Psi(t))$ is the $n$th monic orthogonal polynomial with respect to the varying measure modified by a polynomial term
$$(1-t)d\sigma_{n}^*(t)= \frac{(1-t)d\nu\left(\Psi(t)\right)}{(1-t)^{2n}}.$$

Following the same reasoning, we obtain that
$$
H_{n,n}^*(t) = (1-t)^n h_{1,n}(\Psi(t)),
$$
is the $n$th monic orthogonal polynomial with respect to $d\sigma_{n}^*(t)$. It is not hard to prove that the system $\{\sigma,\{(1-t)^{2n}\},0\}$ is an admissible system, see (\cite{Lago89-1} Def. p 213). Therefore, by~(\cite{Lago89-1} Th. 10), we have
\begin{equation}\label{asym:simple}
\frac{H_{n,n}(t)}{H_{n,n}^*(t)}\unifn \frac{\varphi(t)-\varphi(1)}{{2}(t-1)};\qquad K\subset \CC\setminus [-1,1].
\end{equation}

\begin{theorem}\label{asy:(x+1)}
Under the previous hypothesis we have on compact subsets of $\CC\setminus \RRp$
\begin{align}
\frac{h_{0,n}(z)}{h_{1,n}(z)} & \unifn \left(\frac{z+1}{4}\right)\left(1-\Phi(z)\right),\label{asy:simple1}\\
\frac{h_{v,n}(z)}{h_{w,n}(z)} & \unifn \left(\frac{z+1}{4}\right)^{w-v}\left(1-\Phi(z)\right)^{w-v},\label{asy:simple2}
\end{align}
where $v,w\in \ZZ$.
\end{theorem}

\begin{proof}
From \eqref{asym:simple} and taking the change of variable \eqref{ChangeVariables} we have
\begin{align*}
\frac{h_{0,n}(\Psi(t))}{h_{1,n}(\Psi(t))} & = \frac{(1-t)^n h_{0,n}(\Psi(t))}{(1-t)^n h_{1,n}(\Psi(t))}\\
& = \frac{H_{n,n}(t)}{H_{n,n}^*(t)} \unifn \frac{\varphi(t)-\varphi(1)}{{2}(t-1)} = \frac{{\Phi(z)}-1}{{\Psi^{-1}(z)}-1}.
\end{align*}

To prove $\eqref{asy:simple2}$, note that from Lemma \ref{Lemma:factor(x+1)}.
$$d\mu_k = \left(\frac{x+1}{2}\right)^kd\mu\in \med \text{ if } \mu\in \med.$$

The only hypothesis needed to obtain  \eqref{asy:simple1} is $d\nu\in \med$. Thus if we let now $d\nu = \left(\frac{x+1}{2}\right)^kd\mu = d\mu_{k}$, then $\left(\frac{x+1}{2}\right)d\nu = \left(\frac{x+1}{2}\right)^{k+1} d\mu = d\mu_{k+1}$, where $d\nu \in \med$.

Therefore, $h_{0,n} = \mathfrak{h}_{k,n}$ and $h_{1,n}= \mathfrak{h}_{k+1,n}$, where $\mathfrak{h}_{k,n}$ and $\mathfrak{h}_{k+1,n}$ are the orthogonal polynomials with respect to the measures $d\mu_k$ and $d\mu_{k+1}$, respectively, normalized by having the value $(-1)^k$ at $-1$. Therefore, we have
\begin{equation}
\frac{\mathfrak{h}_{k,n}(z)}{\mathfrak{h}_{k+1,n}(z)} \unifn -\left(\frac{z+1}{4}\right)\left(\Phi(z)-1\right).\label{asy:simple3}
\end{equation}
Without~loss of generality, we can asume $w>v$, otherwise the relation between the measures can be reverted, and~they still belong to $\med$. Stack formula \eqref{asy:simple3} as
$$
\frac{\mathfrak{h}_{v_1,n}(z)}{\mathfrak{h}_{w_1,n}} = \frac{\mathfrak{h}_{v_1,n}(z)}{\mathfrak{h}_{v_1+1,n}}\cdot\frac{\mathfrak{h}_{v_1+1,n}(z)}{\mathfrak{h}_{v_1+2,n}}\cdot \;\cdots \;\cdot \frac{\mathfrak{h}_{w_1-1,n}(z)}{\mathfrak{h}_{w_1,n}},
$$
where $v_1 = v+k$ and $w_1= w+k$. Since the measure $\mu\in \med$, \eqref{asy:simple2} holds.
\end{proof}

%%%%%%%%%%%%%%%%%%%%%%%%%%%%%%%%%%%%%%%%%%%%%%%%%

\section{{Asymptotic  for orthogonal polynomials with respect to a measure modified by a rational~factor}}\label{Ch2-Sec3}

Let $r=\alpha/\beta$, after~canceling out common factors, where
\begin{equation}\label{theFraction}
\begin{aligned}
& \alpha(z)=\prod_{i=1}^{N_1}(z-a_i)^{A_i}, \quad \beta(z)=\prod_{j=1}^{{N_2}}(z-b_j)^{B_j}, \\
& a_i \in \CC\setminus (\RRp\cup \{-1\}),\;  b_j \in \CC\setminus \RRp, \; A_i,B_j\in \NN, \\
&  A=\sum_{i=1}^{N_1} A_{i},\qquad B=\sum_{j=1}^{N_2} B_{j}.
\end{aligned}
\end{equation}

Given a measure $\nu\in \med$, denote by $d\mu(x)= \left(\frac{x+1}{2}\right)^{A-B}  d\nu(x)$ a modified measure, note that according to Lemma \ref{Lemma:factor(x+1)} it holds $\nu \in \med$.

Assume $S_n$ is the  polynomial of least degree not identically equal to zero, such that
\begin{equation}
0= \int_0^{\infty} p(x)S_n(x)r(x)\; d\nu(x), \qquad p\in \PP_{n-1}, \label{InnerProduct01}
\end{equation}
normalized such that $S_n(-1)=(-1)^n$, and~$L_n$ is the $n$th orthogonal polynomial with respect to $d\nu$, normalized such that $L_n(-1)=(-1)^n$. We are interested in the asymptotic behavior of ${S_n/L_n},n\in \ZZp$ in compact subsets of $\CC\setminus \RRp$.

\begin{theorem}\label{FundamentalLemma}
Let $\mu \in \med$ and $\alpha$ and $\beta$ defined as before. Then for all sufficiently large $n$, for~all fixed $d\in \ZZp$, in~compact subsets of $\CC\setminus \RRp$, it holds
\begin{equation}\label{FundResult01}
\frac{S_n(z)}{\ell_{0,n}(z)} \unifn \frac{(-1)^A \alpha(-1)}{4^A(z+1)^{-A}}\prod_{i=1}^{N_1}\left(\frac{\Phi(z)-\Phi(a_i)}{z-a_i}\right)^{A_i}
\prod_{j=1}^{N_2}\left(1-\frac{1}{\Phi(z)\Phi(b_j)}\right)^{B_j}.
\end{equation}
\end{theorem}

\begin{proof}
First we focus on \eqref{InnerProduct01} for $\dsty \alpha(x)=\left(\frac{x+1}{2}\right)^k \beta(x)$ where $k=0,\dots,n-B-1$, we~have
\begin{equation*}
0=\int_{0}^{\infty} \left(\frac{x+1}{2}\right)^k S_n(x)\alpha(x) d\nu(x),
\end{equation*}
{using the change of variables \eqref{ChangeVariables} and considering the expression $d\mu(\Psi(t)) = (1-t)^{B-A} d\nu(\Psi(t))$,} the~previous integral becomes
\begin{align} \label{int_relations}
0=\int_{-1}^1 \left(1-t\right)^{n-B-k-1} (1-t)^{n+A}S_n\left(\Psi(t)\right)\; \alpha\left(\Psi(t)\right)\; \frac{d\mu\left(\Psi(t)\right)}{(1-t)^{2n-1}}.
\end{align}
for $k=0,\dots,n-B-1$. Define the $(n+A)$-degree polynomial $R_{n+A}$ as
\begin{equation*}
R_{n+A}(t):= (1-t)^{n+A}S_n\left(\Psi(t)\right)\; \alpha\left(\Psi(t)\right).
\end{equation*}

Thus, we can consider $d\sigma_n(t) = \frac{d\sigma(t)}{(1-t)^{2n-1}}$ with $d\sigma(t) = d\nu(\Psi(t))$. The~measure $d\sigma_n(t)$  defines a varying orthogonal polynomial system, satisfying Lemma \ref{Lago-1}. We denote by $P_{n,n+A-k}$ the $(n+A-k)$th monic orthogonal polynomial with respect to $d\sigma_n(t)$. According to \eqref{int_relations}, we have the following quasi-orthogonality of order $n-A$
\begin{equation}
R_{n+A}(t):= (1-t)^{n+A}S_n\left(\Psi(t)\right)\; \alpha\left(\Psi(t)\right)\;= \sum_{k=0}^{A+B} \lambda_{n,k} P_{n,n+A-k} (t). \label{linearDecomposition}
\end{equation}

Back to \eqref{linearDecomposition}, we use the connection formula \eqref{ConectionFormula02} and the change of variables \eqref{ChangeVariables} to~obtain
\begin{align}
\left(\frac{2}{z+1}\right)^{n+A} S_n(z)\alpha(z) & = \sum_{k=0}^{A+B} \lambda_{n,k} P_{n,n+A-k}\left(\Psi^{-1}(z)\right)\nonumber\\
& = \sum_{k=0}^{A+B} \lambda_{n,k} \left(\frac{2}{z+1}\right)^{n+A-k}\ell_{A-k,n+A-k}(z),\nonumber\\
S_n(z)\alpha(z)& =  \sum_{k=0}^{A+B} \lambda_{n,k} \left(\frac{z+1}{2}\right)^{k}\ell_{A-k,n+A-k}(z). \label{LinearDecompositionQS}
\end{align}
Observe  that $\lambda_{n,0}= \lambda_0 =(-1)^A\alpha(-1)$ or $S_n$ has  $\deg S_n < n$. Dividing this relation by $\ell_{-B,n-B}$ we get
\begin{equation}
\frac{S_n(z)\alpha(z)}{\ell_{-B,n-B}(z)}= \sum_{k=0}^{A+B} \lambda_{n,k} \left(\frac{z+1}{2}\right)^{k}\frac{\ell_{A-k,n+A-k}(z)}{\ell_{-B,n-B}(z)}.\label{eqOmegaN}
\end{equation}

Set $\dsty \lambda^{**}_{n,k}= \lambda_{n,k}/\lambda_{0}$, $\dsty \lambda^*_n= \left(\sum_{k=0}^{A+B}|\lambda_{n,k}^{**}|\right)^{-1}<\infty$ and introduce the polynomials
$$p_n(z)=\sum_{k=0}^{A+B}\lambda_{n,k}^{**}z^{A+B-k},\qquad p_n^*=\lambda_n^*p_n(z).$$
We will prove that
$$p_n(z)\unifn \hat{p}(z)= \prod^{N_1}_{i=1}\left(z-\frac{\Phi(a_i)}{2}\right) \prod^{N_2}_{j=1}\left(z-\frac{1}{2\Phi(b_j)}\right);\qquad K\subset \CC.$$
To this end, it   suffices to show that
\begin{equation}
p_n^*(z)\unifn c \hat{p}(z)= c\left(z^{A+B}+\lambda_1^{**}z^{A+B-1}+\cdots+\lambda_{A+B}^{**}\right),\label{GoalConvergence}
\end{equation}
where
\begin{equation}
c=\lim_{n\to \infty} \lambda_n^*= \left(\sum_{k=0}^{A+B}|\lambda_{k}|\right)^{-1}. \label{GoalConvergence01}
\end{equation}
 Now, see that $\{p_n^*\}$, for~$n\in \ZZp$ is contained in $\PP_{A+B}$ and the sum of the coefficients of $p_n^*$ for each $n\in \ZZp$, is equal to one. Therefore, this family of polynomials is normal. This means that \eqref{GoalConvergence} can be prove if we check that, for~all $\Lambda\subset \ZZp$ such that
\begin{equation}
\lim_{\substack{n\to\infty\\n\in\Lambda}} p_n^*(z) =p_\Lambda, \label{LimitLambda}
\end{equation}
$p_\Lambda(z)=c\hat{p}(z)$, where $\hat{p}(z)$ and $c$ are defined as above. Since $p_\Lambda\in \PP_{A+B}$ and $p_\Lambda \not\equiv 0$, we can uniquely determine $p_{\Lambda}$ if we find its zeros and leading coefficient. Note that the leading coefficient of $p_\Lambda$ is positive and the sum of the absolute value of its coefficients is one. Therefore, we conclude that the leading coefficient is uniquely determined by the zeros. This automatically implies that  $p_\Lambda(z)=c\hat{p}(z)$ if and only if it is divisible by $\hat{p}(z)$.

The factor $\beta$ is in \eqref{eqOmegaN} and all the zeros of $\ell_{-B,n-B}$ concentrate  on $\RRp$. Thus, we immediately obtain the following $A$ equations, for~$n\geq n_0$:
\begin{equation*}
0= \sum_{k=0}^{A+B} \lambda_n^*\lambda_{n,k}^{**} \left[\left(\frac{z+1}{2}\right)^{k}\left(\frac{\ell_{A-k,n+A-k}}{\ell_{-B,n-B}}\right)\right]^{(v)}(a_i), %\label{AEquations}
\end{equation*}
for $i=1,\dots,{N_1}$ and $v=0,\dots,A_j-1$.

From Lemma \ref{Th-AsympCocient} it follows that, for~compact subsets $K\subset \CC\setminus \RRp$, it holds
\begin{align}
\left[\left(\frac{z+1}{2}\right)^{k}\left(\frac{\ell_{n+A,n+A-k}(z)}{\ell_{-B,n-B}(z)}\right)\right]^{(v)}
\unifn \left[\left(\frac{z+1}{2}\right)^{A+B} \left(\frac{\Phi(z)}{2}\right)^{A+B-k}\right]^{(v)}. \label{AsymptoticEll01}
\end{align}
Relations \eqref{LimitLambda} and \eqref{AsymptoticEll01}, together with the fact that $\Phi$ is holomorphic with $\Phi'\neq 0$ in $\CC\setminus \RRp$, imply, using induction on $v$, that
\begin{equation}
p_{\Lambda}^{(v)}\left(\frac{\Phi(a_i)}{2}\right)=0, \qquad i=1,\dots,N_1,\quad v=0,\dots,A_i-1;\label{polynomial}
\end{equation}

$$p_\Lambda(z)= c\left(\frac{z+1}{2}\right)^{A+B}\sum_{k=0}^{A+B} \lambda_k^{**}\left(\frac{\Phi(z)}{2}\right)^{A+B-k}.$$

On the other hand, take $p(z)=\beta(z)\ell_{-B,n-B}(z)/(z-b_j)^v$ in \eqref{InnerProduct01}, $j=1,\dots, N_2$;   {$v=1,\dots, B_j$. Using \eqref{LinearDecompositionQS} and multiplying by $(v-1)!\frac{\lambda_n^*}{\lambda_0}\tau^2_{n,n-B}$ we have the additional~relations}
\begin{align}
0=& \frac{\lambda_n^*}{\lambda_0}\tau^2_{n,n-B}\int_{0}^\infty \frac{(v-1)!}{(x-b_j)^v}\;\ell_{-B,n-B}(x) S_n(x)\alpha(x) d\nu(x),\nonumber\\
=& \tau^2_{n,n-B}\int_{0}^\infty \frac{(v-1)!}{(x-b_j)^v}\; \ell_{-B,n-B}(x)\; \nonumber \\
& \sum_{k=0}^{A+B} \lambda_n^*\lambda_{n,k}^{**}\left(\frac{x+1}{2}\right)^{k}\ell_{A-k,n+A-k}(x)d\nu(x),\nonumber\\ \nonumber
0 = & \sum_{k=0}^{A+B} \lambda_n^*\lambda_{n,k}^{**} (v-1)!\tau^2_{n,n-B}\\
& \int_0^\infty \left(\frac{x+1}{2}\right)^{k}\frac{\ell_{A-k,n+A-k}(x) \ell_{-B,n-B}(x)}{(x-b_j)^v} d\nu(x),
\label{BEquations}
\end{align}
for each  $b_j$.

Relations \eqref{GoalConvergence}, \eqref{BEquations} and Lemma \ref{LemmaOfIntegrals} %\eqref{LimitDerivatives}
together with the fact that $1/\Phi$ is holomorphic with $(1/\Phi)'\neq 0$ and $1/\sqrt{\left(\psi^{-1}(z)\right)^2-1}\neq 0$ in $\CC\setminus \RRp$, give by induction
\begin{equation*}
p_{\Lambda}^{(v)}\left(\frac{1}{2\Phi(b_j)}\right)=0, \qquad j=1,\dots,N_2,\qquad v=0,\dots,B_j-1.
\end{equation*}

From the previous expression and \eqref{polynomial}  it follows that $p_{\Lambda}$ is divisible by $p_0(z)$. Therefore \eqref{GoalConvergence} and \eqref{GoalConvergence01} hold and
$$p_n(z)\unifn p_0(z), \qquad K\subset \CC.$$

From the previous expression, the~definition of $p_n$, \eqref{eqOmegaN}, \eqref{AsymptoticEll01} with $v=0$,  we obtain
\begin{equation*}
\frac{S_n(z)\alpha(z)}{\ell_{-B,n-B}(z)} \unifn (-1)^A \alpha(-1)\left(\frac{z+1}{2}\right)^{A+B} \hat{p}\left(\frac{\Phi(z)}{2}\right).
\end{equation*}
{Use the asymptotic formula \eqref{AsimpRatioOP} in the previous expression and group conveniently to~obtain}
\begin{multline*}
\frac{S_n(z)}{\ell_{-B,n-B}(z)}\cdot \frac{\ell_{-B,n-B}(z)}{\ell_{0,n}(z)} \unifn
\left(\frac{z+1}{2}\right)^{A}\frac{(-1)^A \alpha(-1)\Phi(z)^{-B}}{\alpha(z)}\\
\prod_{i=1}^{N_1}\left(\frac{\Phi(z)-\Phi(a_i)}{2}\right)^{A_i} \prod_{i=1}^{N_2}\left(\frac{\Phi(z)}{2}-\frac{1}{2\Phi(b_j)}\right)^{B_j}
\end{multline*}
and \eqref{FundResult01} follows for $v=0$. To~prove the formula for $d\in \ZZp$, apply the same technique of the proof of Lemma \ref{Th-AsympCocient}.
\end{proof}

\begin{remark}\
\begin{enumerate}
\item The proof depends on the assumption of $\alpha(-1)\neq 0$, we will remove this restriction in Section~\ref{sec:rationalAsym}.

\item We suppose that $\alpha,\beta$ are monic. We can remove that restriction without loss of generality due to the fact that orthogonal polynomial systems are invariant under the constant modification of~measures.

\end{enumerate}
\end{remark}

Theorem \ref{FundamentalLemma} gives the ratio asymptotic between the orthogonal polynomials with respect to a rational modification of kind $r(x)d\nu(x)$ (a general rational modification with no zeros at $-1$) denoted as $S_n$ and those orthogonal with respect to a modified measure of type $\dsty \left(\frac{x+1}{2}\right)^{A-B}$, denoted as $\ell_{0,n}$.

To obtain the general formula we must find the following limit
$$
\lim_{n\to \infty} \frac{\ell_{0,n}(z)}{L_{n}(z)},
$$
on compact subsets of $\CC\setminus \RRp$, where $L_n(z)$ is the $n$th orthogonal polynomial with respect to $d\nu\in \med$ normalized such that $L_n(-1) = (-1)^n$.

\section{Proof of Theorem \ref{main:result:CH}} \label{sec:rationalAsym}

Now, we obtain an analogous of \eqref{AsymBoundedCase} for measures with support on $\RRp$. Define $\hat{\alpha}$ as
$$\hat{\alpha}(z) = \left(\frac{z+1}{2}\right)^{C}\alpha(z)$$
wherein $\alpha$ is defined in \eqref{theFraction} and $C\in \ZZ_+$ is the multiplicity of the zero $-1$ in $\hat{\alpha}/\beta$. Without~loss of generality we can assume that there are more zeros than poles on $-1$, if~not $C=0$. Also, let $L_n$ be the $n$th orthogonal polynomial with respect to $d\hat{\nu}\in \med$, normalized by the condition $L_n(-1)= (-1)^n$. Denote by $Q_n$ the $n$th orthogonal polynomial with respect to $\hat{r} d \hat{\nu}$, where $r=\hat{\alpha}/\beta$, normalized as usual, $Q_n(-1)=(-1)^n$.

Note that if $C=0$, $\hat{r} = r$ and $Q_n = S_n$, as~defined in Section~\ref{Ch2-Sec3}. Under~this notation, \eqref{final:formula}~is written~as

\begin{equation*}\label{final:formula-2}
\frac{Q_n^{(d)}(z)}{L_n^{(d)}(z)} \unifn \left(\frac{2i}{\sqrt{z}+i}\right)^C\prod_{i=1}^{N_1}\left(\frac{\sqrt{a_i}+i}{\sqrt{z}+\sqrt{a_i}}\right)^{A_i}\prod_{j=1}^{N_2} \left(\frac{\sqrt{z}+\sqrt{b_j}}{\sqrt{b_j}+i}\right)^{B_j},
\end{equation*}
in compact subsets  of $\CC\setminus \RR$, for~$d\in \ZZp$.

\begin{proof}[Proof of Theorem \ref{main:result:CH}]\

 Let us first observe that $Q_n$ is orthogonal with respect to $\left(\frac{x+1}{2}\right)^{C}\frac{\alpha}{\beta}d\hat{\nu}$. Then if we set
\begin{equation}\label{measure:relation}
d\hat{\nu} = \left(\frac{x+1}{2}\right)^{-C}d\nu,
\end{equation}
we obtain that $Q_n$ is orthogonal with respect to $\frac{\alpha}{\beta}d\nu$, and~satisfies the hypotheses of \mbox{Theorem~\ref{FundamentalLemma}}, thus we have on compact subsets of $\CC\setminus \RRp$
$$
\frac{Q_n(z)}{\ell_{0,n}(z)}\unifn \mathfrak{F}(z),
$$
where $\mathfrak{F}(z)$ is given in \eqref{FundResult01}.

On the other hand, $\ell_{0,n}$ is orthogonal with respect to $\left(\frac{x+1}{2}\right)^{A-B}d\nu$. This means by \eqref{measure:relation} that $\ell_{0,n}$ is orthogonal with respect to $\left(\frac{x+1}{2}\right)^{A-B+C}d\hat{\nu}$.
Thus, taking into account \mbox{Theorem \ref{asy:(x+1)}}, we have
$$
\frac{\ell_{0,n}(z)}{L_n(z)}\unifn \left(\frac{z+1}{4}\right)^{B-A-C}\left(1-\Phi(z)\right)^{B-A-C}.
$$
Multiply the expressions corresponding to
\begin{align}\label{consider}
\left(\frac{z+1}{4}\right)^{B-A-C}\left(1-\Phi(z)\right)^{B-A-C}\cdot  \mathfrak{F}(z),
\end{align}

Let us break down this expression into the following~terms

\begin{align*}
\mathfrak{F}(z) & = \frac{(-1)^A \alpha(-1)}{4^A(z+1)^{-A}}\prod_{i=1}^{N_1}\left(\frac{\Phi(z)-\Phi(a_i)}{z-a_i}\right)^{A_i}
\prod_{i=1}^{N_2}\left(1-\frac{1}{\Phi(z)\Phi(b_j)}\right)^{B_j}.\\
(-1)^A \alpha(-1) & = \prod_{i=1}^{N_1}(1+a_i)^{A_i}\\
(1-\phi(z)) & = -\frac{2i}{\sqrt{z}-i}\\
\frac{\Phi(z)-\Phi(a_i)}{z-a_i} & = \frac{-2i}{(\sqrt{z}-i)(\sqrt{a_i}-i)(\sqrt{a_i}+\sqrt{z})}\\
1-\frac{1}{\Phi(z)\Phi(b_j)} & = \frac{2i\left(\sqrt{b_j}+\sqrt{z}\right)}{(\sqrt{z}+i)\left(\sqrt{b_j}+i\right)}.
\end{align*}

On the other hand
\begin{align*}
\prod_{i=1}^{N_1}\left(\frac{\Phi(z)-\Phi(a_i)}{z-a_i}\right)^{A_i} & = \left(\frac{-2i}{\sqrt{z}-i}\right)^A \prod_{i=1}^{N_1} \left(\frac{1}{(\sqrt{a_i}-i)(\sqrt{a_i}+\sqrt{z})}\right)^{A_i}\\
\prod_{j=1}^{N_2}\left(1-\frac{1}{\Phi(z)\Phi(b_j)}\right)^{B_j} & = \left(\frac{2i}{\sqrt{z}+i}\right)^B\prod_{j=1}^{N_2}\left( \frac{\sqrt{b_j}+\sqrt{z}}{\sqrt{b_j}+i}\right)^{B_j}
\end{align*}

Combining these terms in \eqref{consider} we obtain
\begin{align*}
& \left(\frac{z+1}{4}\right)^{B-A-C}\left(1-\Phi(z)\right)^{B-A-C}\cdot  \mathfrak{F}(z)\\
= & \frac{1}{4^A}\prod_{i=1}^{N_1}(1+a_i)^{A_i}\left(\frac{-2i}{\sqrt{z}-i}\right)^{B-C-A}\left(\frac{z+1}{4}\right)^{B-C} \left(\frac{-2i}{\sqrt{z}-i}\right)^A \left(\frac{2i}{\sqrt{z}+i}\right)^B\\
& \cdot \prod_{i=1}^{N_1} \left(\frac{1}{(\sqrt{a_i}-i)(\sqrt{a_i}+\sqrt{z})}\right)^{A_i} \prod_{j=1}^{N_2}\left( \frac{\sqrt{b_j}+\sqrt{z}}{\sqrt{b_j}+i}\right)^{B_j}.
\end{align*}

Finally, taking into account
\begin{align*}
 \prod_{i=1}^{N_1} \left(\frac{\sqrt{a_i}+i}{\sqrt{a_i}+\sqrt{z}}\right)^{A_i} & = \prod_{i=1}^{N_1}(1+a_i)^{A_i}\cdot \prod_{i=1}^{N_1} \left(\frac{1}{(\sqrt{a_i}-i)(\sqrt{a_i}+\sqrt{z})}\right)^{A_i}\\
\left(\frac{2i}{\sqrt{z}+i}\right)^{C} & =  \frac{1}{4^A}\left(\frac{-2i}{\sqrt{z}-i}\right)^{B-C-A}\left(\frac{z+1}{4}\right)^{B-C} \left(\frac{-2i}{\sqrt{z}-i}\right)^A \left(\frac{2i}{\sqrt{z}+i}\right)^B
\end{align*}
we obtain \eqref{final:formula} for $d=0$.  To~prove \eqref{final:formula} for $d\geq 1$, use induction in $d$ and the method from the proof of Lemma \ref{Th-AsympCocient}. The~proof is complete.
\end{proof}

\section*{Author's Note}
This paper is an updated version of the original article incorporating several minor corrections: C. Fel\'{\i}z-S\'{a}nchez, H. Pijeira-Cabrera, and J. Quintero-Roba, \emph{Asymptotics for Orthogonal Polynomials with Respect to a Rational Modification of a Measure Supported on the Semi-Axis}, Mathematics \textbf{2024}, \emph{12}(7), 1082.

%%%%%%%%%%%%%%%%%%%%%%%%%%%%%%%%%%%%%%%%%%%%%%%%%%%%%%%%%%%%%%%%%%%%%%%%%%%%%%%%%%%
 
%%%%%%%%%%%%%%%%%%%%%%%%%%%%%%%%%%%%%%%%%%%%%%%%%%%%%%%%%%%%%%%%%%%%%%%%%%%%%%%%%

\end{document}